\documentclass[12pt,reqno]{amsart}

\usepackage{amsmath, amsfonts, amsthm, amssymb, color,  graphicx, mathrsfs, cite}
\usepackage{stmaryrd}

\makeatletter
\@addtoreset{figure}{section}
\makeatother
\theoremstyle{plain}
\newtheorem{Thm}{Theorem}[section]
\newtheorem{Lemma}{Lemma}[section]
\newtheorem{Proposition}{Proposition}[section]

\theoremstyle{definition}

\theoremstyle{Remark}
\newtheorem{Remark}{Remark}[section]

\numberwithin{equation}{section}
\allowdisplaybreaks

\usepackage{ifpdf}
\ifpdf \usepackage[colorlinks=true, citecolor=blue, linkcolor=blue, urlcolor=blue]{hyperref} \fi

\newcommand{\pt}{\partial_t}

\newcommand{\pxi}{\partial_i}

\newcommand{\pxk}{\partial_k}

\newcommand{\Do}{\mathbf{D}_0}
\newcommand{\Dn}{\mathbf{D}_{\neq}}

\newcommand{\mr}{\mathring}
\newcommand{\ac}{\acute}

\begin{document}
	
\title[composite wave]
{Nonlinear stability of the composite wave of planar rarefaction waves and planar contact waves for viscous conservation laws with non-convex flux under multi-dimensional periodic perturbations.}

\author{Meichen Hou}
\address{Center for nonlinear studies, School of Mathematics, Northwest University, Xi'an 710069, P.R.China.}
\email{meichenhou@nwu.edu.cn}

\author{Lingda Xu}
\address{Department of Mathematics, Yau Mathematical Sciences Center, Tsinghua University, Beijing 100084, P.R.China.\\
Yanqi Lake Beijing Institute of Mathematical Sciences And Applications, Beijing 101408, P.R.China}
\email{xulingda@tsinghua.edu.cn}

\date{} \maketitle

\begin{abstract}
	
	In this paper, we study the nonlinear stability of the composite wave consisting of planar rarefaction waves and planar contact waves for viscous conservation laws with degenerate flux under multi-dimensional periodic perturbations. To the level of our knowledge,  it is the first stability result of the composite wave for conservation laws in several dimensions. Moreover, the perturbations studied in the present paper are periodic, which keep constantly oscillating at infinity. Suitable ansatz is constructed to overcome the difficulty caused by this kind of perturbation and delicate estimates are done on zero mode and non-zero mode of perturbations. We obtain satisfactory decay rates for zero mode and exponential decay rates for non-zero mode.
	
	\bigbreak

	
\end{abstract}

\section{Introduction }
\setcounter{equation}{0}

In this paper, we study the Cauchy problem for the following equations

\begin{equation}\label{eq}
\left\{
\begin{aligned}
&{\partial_{t} u+\sum\limits_{i=1}^{n}\partial_{x_i}(f_i(u))=\sum\limits_{i, j=1}^{n} a_{i j} u_{x_{i} x_{j}}}, \ {(t>0, x \in \mathbb{R}^{n}),}\\
&{u(x, 0)=u_{0}(x),}\ x\in\mathbb{R}^n,
\end{aligned}
\right.
\end{equation}
where the unknown function $u\in \mathbb{R}^1$ is scalar, the viscosity matrix $A\equiv(a_{ij})\in\mathbb{R}^{n\times n}$ is a constant positive definite matrix, $f_i(u)(i=1,2,..,n)$ are the smooth flux functions.%initial data $u_0(x)$ is a given function which will be determined later. constants $u_{\pm}$ are so-called far field states,  And $\mathbb{T}^{n-1}\equiv\left(\mathbb{R}/\mathbb{Z}\right)^{n-1}$ is (n-1)-dimensional torus.

In (\ref{eq}), we further assume

	\begin{equation}\label{fi}
	\begin{aligned}
	 f_1\in C^{3}(\mathbb{R}),\ \ \ \ f_1(0)=f_1'(0)=0.
	\end{aligned}
	\end{equation}

We will examine the large-time behavior of the global solution for $(\ref{eq})$, which has a lot of similarities with the Riemann problem for the corresponding equations without viscosity, i.e.,
\begin{align}\label{eq1}
\begin{cases}
\partial_{t} u+\sum\limits_{i=1}^{n}\partial_{x_i}(f_i(u))=0\ {(t>0, x \in \mathbb{R}^{n})}, \\
{u(x, 0)=u^{(1)}_{0}(x)}, &
\end{cases}
\end{align}
where

\begin{align*}
u^{(1)}_{0}(x)=\begin{cases}
u_-,&x_1<0,\\
u_+,&x_1>0.
\end{cases}\quad (u_-<u_+)
\end{align*}

There has been a wealth of research into the stability of fundamental wave patterns for viscous conservation laws. For the one-dimensional (1-d) case, for example, in 1960, \cite{IKO,IK} studied the asymptotic stability of solutions when $n=1$ in (\ref{eq}) and the flux $f_1$ is strictly convex, of which only one single wave patterns generated. The convergence rate was obtained in \cite{KM1985,MN1994} if initial data belongs to a weighted Sobolev space, and the restrictions of initial data have been relaxed in \cite{HX2022}.  An interesting  $L^1$ stability theorem was established in \cite{FS1998}. Considering the general case that the flux $f_1$ is not uniformly genuinely nonlinear, there are several wave patterns in the Riemann solution. And we refer to Matsumura-Yoshida\cite{MY} and Yoshida\cite{Y}\cite{Y2}, who has studied the asymptotic behavior of superpositions of rarefaction waves and contact waves.

For the multi-dimensional (m-d) case, Z.P. Xin\cite{X} showed that the planar rarefaction wave for viscous conservation laws in several dimensions is stable, and there are several interesting extensions to this result, see \cite{I,KNN}. Shi-Wang\cite{SW} proved the stability theorem of the viscous shock wave, Kang-Vasseur-Wang proved the $L^2$-contraction of large planar shock waves for m-d scalar viscous conservation laws, we refer to \cite{V1,V2}. For the stability of planar shock wave, \cite{Z1}
revealed that the nonlinear stability of viscous shocks can be implied by the spectral
stability, where the latter one is somehow equivalent to the linearized stability with
respect to zero-mass perturbations, see \cite{Z2}.
Another interesting and important problem is considering the asymptotic stability of Riemann solutions under periodic perturbations for conservation laws. The research of this problem is started by Lax \cite{Lax} and Glimm-Lax \cite{GL} and were extended by \cite{XYY1,XYY2,YY,HY}.   

Our paper concerns the stability of waves of different types compounded together under m-d periodic perturbations. We use the new weight function $\eta$ for composite wave and thus succeed in constructing an ansatz which can overcome the difficulties posed by the constant oscillation of the initial perturbation at infinity. Motivated by \cite{Yuan1,Yuan2}, we decomposite the perturbation into zero and non-zero mode, and satisfactory decay rates are obtained for both components. Specifically, for the zero mode, we obtain the same decay rate as the 1-d case, see Theorem 1.4 in \cite{Y}, but our initial perturbations are not integrable, which is the key condition in \cite{Y}. Without this condition, the decay rate of $L^{\infty}$-norm is only $(1+t)^{-1/4+\epsilon}$. And we proved that the non-zero mode decays exponentially with respect to time $t$.

This paper is organized into the following structure. In section 2, we mainly introduce some properties of the planar waves and give the construction of the ansatz. Then the stability results are stated in Theorem \ref{mt}. In section 3, we rearranged our problem to a new perturbation $\phi,$  then we divided $\phi$ into the zero modes $\mr{\phi}$ and non-zero mode $\ac{\phi}$. The priori estimates for those two modes are listed in proposition \ref{pet1}-\ref{pet2}. Finally, in section 4, we mainly prove our results.

\textbf{Notations:}
The whole domain $\mathbb{R}\times\mathbb{T}^{n-1}$ is  abbreviated as $\Omega.$
$\|f(\cdot,t)\|_{L^p}$ and $\|f(\cdot,t)\|_{H^k}$ denote the norms of usual Lebesgue space $L^p$ and Sobolev space $H^k$ on the whole domain $\Omega$. $\|f(\cdot,t)\|_{L^p_{x_i}}$ and $\|f(\cdot,t)\|_{k,x_{i}}$ denote the norms of corresponding space on the $x_i$-direction. $C,C_i$ denotes the  generic positive constant which is independent of time $t$  unless otherwise stated.
Sometimes the space variable $x$ is denoted as  $x=(x_1,x')$, where  $x'=(x_2,x_3,...,x_n)$. Moreover,$dx'=(dx_2,dx_3,...dx_n).$
\section{Ansatz and main results}

In this section, we introduce some properties about the multi-wave patterns $\bar{U}$. Then we construct the ansatz and state the stability results in theorem \ref{mt}.
%Specifically, We are going to cover three different aspects of this topic in this section. The first thing we will do is introducing some asymptotic properties of two wave patterns, see lemma \ref{rl}-lemma \ref{xl}. %Furthermore, since $\hat{u}$ contains multiple wave patterns, we should estimate the nonlinear interactions between different wave patterns. There is no separation between different wave patterns in contrast to how it is in systems. To overcome this difficulty, we estimate the curve $X(t)$ defined in (\ref{X}) in the lemma \ref{xl}.
%Secondly, to study the multi-wave $\hat{u}$ under a n-d periodic perturbation, we need to construct the ansatz $\bar{u}$ which is periodic in $x_i(i=2,...,n)$ direction, see \eqref{ansatz}. At last, the stability results will be shown in theorem \ref{mt}.
\subsection{The composite wave patterns}
\ \\

A planar wave (in $x_1$-direction, without loss of generality) is a solution of the following Cauchy problem:
\begin{align}\label{eqr}
\begin{cases}
\partial_{t} u+\partial_{x_1}(f_1(u))=0 & {(t>0, x_1 \in \mathbb{R})},\\
{u(x_1, 0)=u^{(2)}_{0}(x_1)},
\end{cases}
\end{align}
where
\begin{align}\label{in}
u^{(2)}_{0}(x_1)=\begin{cases}
u_-,&x_1<0,\\
u_+,&x_1>0.
\end{cases}
\end{align}
With the results in Liu\cite{L}, Matsumura-Nishihara\cite{MN1}, we study the solution of \eqref{eqr}
with smooth initial data, which converges to planar rarefaction wave in $L^\infty$-norm as $t\rightarrow\infty$. Specifically,
\begin{align}\label{eqrs}
\begin{cases}
\partial_{t} U+\partial_{x_1}(f_1(U))=0 \quad {(t>0, x_1 \in \mathbb{R})}, \\
{U(x_1, 0)=U^{r}_{0}(x_1)}=\frac{u_{+}+u_{-}}{2}+\frac{u_{+}-u_{-}}{2} \tanh x_1.
\end{cases}
\end{align}
As a result of our study, we explore a more general case, namely, there exists an interval $\equiv(a,b)\subset\mathbb{R}$, which exists in such a way that
\begin{align}\label{f1}
\left\{\begin{array}{ll}{f_1^{\prime \prime}(u)>0} & {(u \in(-\infty, a] \cup[b,+\infty))}, \\
{f_1^{\prime \prime}(u)=0} & {(u \in(a, b))}.\end{array}\right.
\end{align}
There are some theories about studying (\ref{eqr}) under 1-d condition (\ref{in}) and \eqref{f1}, see \cite{JS} for example. It is known that the Riemann solution consists of rarefaction waves and contact discontinuities in this case. The explicit formulas of them depend on $a,\ b,\ u_-$, and $u_+$.  Therefore, we need to discuss different situations separately to avoid confusion.

We denote the planar rarefaction wave connecting end states $u_-$ and $u_+$ by $U^r(\frac{x}{t};u_-,u_+)$, on which $u_{\pm}$ are two constants $(u_{\pm} \in(-\infty, a] \cup[b,+\infty)))$. The explicit formula of $U^{r}(\frac{x}{t};u_-,u_+)$ is
\begin{align}\label{r}
u=U^{r}\left(\frac{x_1}{t} ; u_{-}, u_{+}\right)\equiv\begin{cases}
{u_{-}}, & {\left(x_1 \leq f_1'\left(u_{-}\right) t\right),} \\
{(f_1')^{-1}\left(\frac{x_1}{t}\right)}, & {\left(f_1'\left(u_{-}\right) t \leq x_1 \leq f_1'\left(u_{+}\right) t\right),} \\
{u_{+}}, & {\left(x_1 \geq f_1'\left(u_{+}\right) t\right)}.
\end{cases}
\end{align}

And the viscous contact wave (or so-called viscous version of contact discontinuity) connecting $v_-$ and $v_+$ $(v_{\pm}\in[a,b])$ is denoted by $U^{c}\left(\frac{x_1-\lambda t}{\sqrt{t}} ; v_{-}, v_{+}\right),$ that is
\begin{align}\label{c}
u=U^c\left(\frac{x_1-\lambda t}{\sqrt{t}} ; v_{-}, v_{+}\right)\equiv v_{-}+\frac{v_{+}-v_{-}}{\sqrt{\pi}} \int_{-\infty}^{\frac{x_1-\lambda t}{\sqrt{4 a_{11} t}}} \mathrm{e}^{-\xi^{2}} \mathrm{d} \xi,
\end{align}
where $\lambda\equiv\frac{f_1(a)-f_1(b)}{a-b}$. $U^c$ is the solution of the following heat equation
$$\partial_{t} U^c+ \lambda \partial_{x_1} U^c=a_{11} \partial_{x_1}^{2} U^c.$$

Now we can list the asymptotic attractors of different cases, which is denoted by $\bar{U}$,

i) $(a,b)\cap(u_-,u_+)=\emptyset$, the asymptotic attractor is

$$\bar{U}=U^{r}\left(\frac{x_1}{t} ; u_{-}, u_{+}\right);$$
There are only a rarefaction wave. We omit this case since it is the same as case of which $f_1(u)$ is genuinely nonlinear;

ii) $a<u_-<b<u_+$, the asymptotic attractor is
$$\bar{U}=U^c\left(\frac{x_1-\lambda t}{\sqrt{t}} ; u_{-}, b\right)+U^{r}\left(\frac{x_1}{t} ; b, u_{+}\right)-b;$$

iii) $u_-<a<u_+<b$, the asymptotic attractor is
$$\bar{U}=U^{r}\left(\frac{x_1}{t} ; u_-, a\right)+U^c\left(\frac{x_1-\lambda t}{\sqrt{t}} ; a, u_+\right)-a;$$

iv) $u_-<a<b<u_+$, the corresponding asymptotic attractor is
$$\bar{U}=U^{r}\left(\frac{x_1}{t} ; u_-, a\right)-a+U^c\left(\frac{x_1-\lambda t}{\sqrt{t}} ; a, b\right)+U^{r}\left(\frac{x_1}{t} ; b, u_+\right)-b.$$

\subsection{Some properties for the planar waves}
There is the explicit formula given in the two references (\ref{r}) and (\ref{c}), we will study more properties surrounding these two profiles in the coming sections. Among these profiles, the first is the smooth approximation of the rarefaction wave $U^r$, which is denoted as $u^r$. As a starting point, let us consider the following initial value problem. It is possible to denote the solution with the far field states $(w_-,w_+)$ as $w(x_1,t;w_-,w_+)$.
\begin{align}\label{sr}
\begin{cases}
{\partial_{t} w+\partial_{x_1}\left(\frac{1}{2} w^{2}\right)=0 \quad(t>0, x_1 \in \mathbb{R})}, \\
w(x_1, 0)=\frac{w_{+}+w_{-}}{2}+\frac{w_{+}-w_{-}}{2} \tanh x_1 \quad(x_1 \in \mathbb{R}).
\end{cases}
\end{align}
Our profile $u^r(x_1,t;u_-,u_+)$ is defined as

\begin{equation}\label{ur}
\begin{aligned}
u^{r}\left(x_1,t ; u_{-}, u_{+}\right):=(f_1')^{-1}\left(w\left(x_1,t; (f_1')^{-1}(u_-), (f_1')^{-1}(u_+)\right)\right).
\end{aligned}
\end{equation}
With direct calculation, we find that, $u^r$ satisfies
\begin{align}
\begin{cases}
u^r_t+(f_1(u^r))_{x_1}=0,\\
u^r(x_1,0)=(f_1')^{-1}(U^r_0(x_1)),
\end{cases}
\end{align}
where $U^r_0(x_1)$ is defined in \eqref{eqrs}, and $\lim\limits_{x_1\rightarrow\pm\infty}u^r(x_1,t)=u_{\pm}$. Moreover,

\begin{equation*}
\lim_{t\rightarrow\infty}\sup_{x_1\in\mathbb{R}}|U^r(\frac{x_1}{t};u_-,u_+)-u^r(x_1,t;u_-,u_+)|=0.
\end{equation*}

Many works study the smooth approximation of rarefaction waves, we refer to\cite{L}\cite{MN1}\cite{MY} for more details. Here we list the properties as the following lemma.
\begin{Lemma}[Decay properties of $u^r$]\label{rl}
	Under the assumptions \eqref{fi},\eqref{f1} and $u_-<u_+$, we have the following estimates:
	
	(1)  $u_-<u^r(x_1, t)<u_+$, and $\partial_{x_{1}} u^r>0$.
	
	(2) For $1\leq p\leq\infty$, there exist a positive constant $C(p,u_-,u_+)$ depending on $p,u_-,u_+$ such that
	\begin{align}\label{rp1}
	\begin{cases}
	{\left\|\partial_{x_1} u^{r}(t)\right\|_{L^{p}(\mathbb{R})} \leq C(p,u_-,u_+)(1+t)^{-1+\frac{1}{p}}},{(t \geq 0)},\\
	{\left\|\partial_{x_1}^{2} u^{r}(t)\right\|_{L^{p}(\mathbb{R})} \leq C(p,u_-,u_+)(1+t)^{-1}}.
	\end{cases}
	\end{align}
	Especially, when $p=\infty$, we have
	\begin{equation}
	\begin{aligned}
	\sup_{x_1\in\mathbb{R}}|\partial_{x_{1}}u^r(\cdot,t)|\leq C(u_-,u_+)(1+t)^{-1}.
	\end{aligned}
	\end{equation}

	(3) For any $\delta\in(0,1)$, there exists a positive constant $C_\delta$ such that the following inequalities hold,
	\begin{align}
	&\left|u^{r}(x_1,t)-u_{+}\right| \leq C_{\delta}(1+t)^{-1+\delta} \mathrm{e}^{-\delta\left|x_1-\lambda_{+} t\right|} \quad\left(t \geq 0, x_1 \geq \lambda_{+} t\right),\\
	&\left|u^{r}(x_1,t)-u_{-}\right| \leq C_{\delta}(1+t)^{-1+\delta} \mathrm{e}^{-\delta\left|x_1-\lambda_{-} t\right|} \quad\left(t \geq 0, x_1 \leq  \lambda_{-} t\right),\\
	&\left|u^{r}(x_1,t)-U^{r}\left(\frac{x_1}{t}\right)\right| \leq C_{\delta}(1+t)^{-1+\delta} \quad\left(t \geq 1, \lambda_{-} t \leq x_1 \leq \lambda_{+} t\right),
	\end{align}

	where $\lambda_{\pm}=(f_1')^{-1}(u_\pm)$.
\end{Lemma}
In lemma \ref{rl}, we have introduced some decay properties of $u^r(x_1,t)$. Next, we will study the viscous contact wave  $u^c(x_1,t).$ Recall that $U^c(x_1,t;v_-,v_+)$ defined by (\ref{c}) satisfies the Cauchy problem
\begin{align}
\partial_{t} U^c+ \lambda \partial_{x} U^c=a_{11} \partial_{x}^{2} U^c.
\end{align}
Notice that it is a parabolic equation, so we know that $U^c(x_1,t)\in C^\infty(\mathbb{R}\times(0,\infty))$. But when $t\rightarrow 0$, $U^c$ is no longer continuous. To avoid the singularity, we consider

\begin{equation}\label{uc}
\begin{aligned}
u^c(x_1,t;v_-,v_+)\equiv U^{c}(\frac{x_1}{\sqrt{1+t}},v_-,v_+).
\end{aligned}
\end{equation}

Here, we only consider the essential case $\lambda=0$, since we can do a transformation to obtain the other cases. The explicit formula of $u^c$ is given by (\ref{c}), so we can obtain the following properties by direct calculations.
\begin{Lemma}\label{cl}
	Under the assumptions \eqref{fi}, \eqref{f1}, and $v_-<v_+$, the following properties hold
	
	(1) $\lim\limits_{t\rightarrow\infty}\sup_{x_1 \in \mathbb{R}}|U^c(x_1,t)-u^c(x_1,t)|=0.$
	
	(2) $u_-<u^c(x_1,t)<u_+$, and $\partial_{x_1} u^c(x_1,t)>0$,
	
	(3) For $1\leq p\leq \infty$, there exists a positive constant $C(p,u_-,u_+)$, such that
	\begin{align}
	\|\partial_{x_1} u^c(\cdot,t)\|_{L^p(\mathbb{R})}\leq C(p,u_-,u_+) t^{-\frac{1}{2}(1-\frac{1}{p})},\ \ \ \ \ \ (t>0).
	\end{align}
\end{Lemma}
Because the asymptotic attractor $\bar{U}$ contains multiple waves, we should study the interactions between these two wave patterns. Different from the cases of systems, in the scalar case, different wave patterns will not separate from each other, so the estimate of interaction is more difficult.

Observing lemma \ref{rl} and \ref{cl}, we know that the curve connecting two wave patterns is the key point. For any $t>0$, we set $X(t)\in\mathbb{R}$ as the curve connecting two wave patterns, that is, $X(t)$ satisfies
\begin{align}\label{X}
\hat{u}(X(t),t)=u^c(X(t),t)+u^r(X(t),t)=0,
\end{align}
where $u^c(X(t),t)=u^c(X(t),t;u_-,0)$, $u^r(X(t),t)=u^r(X(t),t;0,u_+)$, $u_-<0<u_+$. Then we have the following
\begin{Lemma}[\cite{MY}] \label{xl}
	Assume $u_-<0<u_+$, the following properties hold
	
	(1) There exists a positive $T_0$, such that, for any $t>T_0$,
	\begin{align}
	\sqrt{4a_{11}(1+t)}\leq X(t)\leq \lambda_{+} (1+t).
	\end{align}
	
	(2) Define $X(t)$ as in (\ref{X}), for $t>T_0$, we have
	\begin{align}
	\left|(f_1')^{-1}\left(\frac{X(t)}{1+t}\right)-\frac{\left|u_{-}\right|}{\sqrt{\pi}} \int_{\frac{X(t)}{\sqrt{4 a_{11}(1+t)}}}^{\infty} \mathrm{e}^{-\xi^{2}} \mathrm{d} \xi\right| \leq C(1+t)^{-\frac{3}{4}}.
	\end{align}

	(3) For any positive constant $\delta\in(0,1)$, there exists a constant $C_\delta>0$ such that
	
	\begin{align}
	\left(C_\delta+ln(1+t)^{\frac{1}{2(1+\delta)}}\right)^{\frac{1}{2}}\sqrt{1+t}\leq X(t)\leq \left(C+ln(1+t)^{\frac{1}{2}}\right)^{\frac{1}{2}}\sqrt{1+t},
	\end{align}
	where $ln\ t= \log_e (t)$.
\end{Lemma}

\subsection{The initial data and the construction of ansatz}

In section 1, we introduce some cases of asymptotic attractors see i), ii), iii), iv). In this section, we will further reduce the case and reformulate the problem to avoid some unnecessary discussion in \cite{MY}.

Firstly, with the following transformations
\begin{align}
&x_1-\lambda t \mapsto x_1,\ \ \  u-b \mapsto u,\\
&f_1(u+b)-f_1^{\prime}(b) u-f_1(a) \mapsto f_1(u),\ \ \ a-b \mapsto a.
\end{align}
we reduce (\ref{f1}) to
\begin{align}
\left\{\begin{array}{ll}{f_1^{\prime \prime}(u)>0} & {(u \in(-\infty, a] \cup[0,+\infty))},\\
{f_1(u)=0} & {(u \in(a, 0))}.
\end{array}\right.
\end{align}

Secondly, two cases in ii), and iii) are similar, so we only consider case ii) $a<u_-<0<u_+$. Next, comparing the difference between ii) and iv), we know that the extra terms are nonlinear interactions between two rarefaction waves separated from each other. Since the interactions between rarefaction waves and contact waves are more difficult, we only consider the essential case ii).

	Furthermore, as pointed out in\cite{MY}, we can assume $a=-\infty$ since the proof is almost the same as the case that $a$ is a finite number. Thus, in this paper, we treat the following case
\begin{align}\label{f''}
\left\{\begin{array}{ll}{f_1^{\prime \prime}(u)>0} & {(u \in[0, \infty))},\\
{f_1(u)=0} & {(u \in(-\infty, 0))}.
\end{array}\right.
\end{align}

In this case, the asymptotic attractor is
\begin{align}
\bar{U}\equiv U^{c}(\frac{x_1}{\sqrt{t}};u_-,0)+U^{r}(\frac{x_1}{t};0,u_+).
\end{align}

Define the multi-planar wave $\hat{u}(x_1,t)$ as
\begin{align}\label{hatu}
\hat{u}(x_1,t)\equiv U^{c}(\frac{x_1}{\sqrt{1+t}};u_-,0)+U^r(\frac{x_1}{1+t};0,u_+).
\end{align}

For simplicity, we use the following notations
$$
u^c(x_1,t)\equiv U^{c}(\frac{x_1}{\sqrt{1+t}};u_-,0),\ \ \ \ \ u^r(x_1,t)\equiv u^r(\frac{x_1}{1+t};0,u_+).
$$

We want to consider the Cauchy problem \eqref{eq} with the following initial data
\begin{equation}\label{equ}
u(x,0)=\hat{u}(x_1,0)+V_0(x),\ (\hat{u}(x_1,0)=\hat{u}_0(x_1)\rightarrow u_{\pm},\ x_1\rightarrow\pm\infty).
\end{equation}
Here the initial data $u(x,0)$ can be regarded as a small periodic perturbation around $\hat{u},$ that is
\begin{equation}
V_0(x)\in H^{[\frac{n}{2}]+2}(\mathbb{T}^{n}), \ \|V_0\|_{H^{[\frac{n}{2}]+2}(\mathbb{T}^{n})}\leq \varepsilon,\ \int_{\mathbb{T}^{n}}V_0(x)dx=0.
\end{equation}
Set $\bar{u}_{\pm}(x,t)$ be the solution of system \eqref{eq} with the following periodic initial data
\begin{equation}
\bar {u}_{\pm}(x,0)=u_{\pm}+V_0(x),%\ u^0(x,0)=V_0(x,0),
\end{equation}
respectively. Then we have following lemma:
\begin{Lemma}\label{lem-periodic-solution}
	For the scalar conservation laws \eqref{eq} with the following periodic initial data
	\begin{equation}\label{ip}
	\begin{aligned}
	\bar{u}_{\pm}(x,0)=u_{\pm}+V_0(x),\  x\in\mathbb{R}^{n},
	\end{aligned}
	\end{equation}
	where
	\begin{equation}\label{cd-zeroav}
	\int_{\mathbb{T}^n}V_0(x)dx=0.
	\end{equation}
	Then there exists a constant $\varepsilon_0>0$ such that for $\varepsilon:=\|V_0\|_{H^{[\frac{n}{2}]+2}(\mathbb{T}^{n})}\leq \varepsilon_0,$ the Cauchy problems \eqref{eq} with \eqref{ip} admits a pair of unique global periodic solutions
	$\bar{u}_{\pm}(x,t)\in C([0,+\infty);H^{[\frac{n}{2}]+2}(\mathbb{T}^n))$ satisfying
	\begin{equation}
	\int_{\mathbb{T}^n}(\bar{u}_{\pm}-u_{\pm})(x,t)dx=0,\quad  t\geq 0,
	\end{equation}
	and
	\begin{equation}\label{pe-de}
	\|\bar{u}_{\pm}-u_{\pm}\|_{W^{1,\infty}(\mathbb{R}^n)}\leq C\|V_0\|_{H^{[\frac{n}{2}]+2}(\mathbb{T}^{n})} e^{-\bar{c}t}, \ t\geq 0,
	\end{equation}
	where $\bar{c}$ is a positive constant independent of $t$.
\end{Lemma}
\begin{proof}
This lemma can be proved by using standard $L^2$ energy estimates and the Poincare's inequality, we omit the details. 
\end{proof}
%\begin{Remark}
%Similarly, from lemma \ref{lem-periodic-solution}, \eqref{eq} adimit a global periodic solution $u^0(x,t)$with initial data $u^0(x,0).$ Moreover,

%\begin{equation}
%\begin{aligned}
%\int_{\mathbb{T}^n}{u}^0(x,t)dx=0, t\geq 0,\ 	\|u^0(\cdot,t)\|_{W^{k,\infty}(\mathbb{R}^n)}\lesssim \varepsilon e^{-ct}, \ t>0.
%\end{aligned}
%\end{equation}
%\end{Remark}
In order to study the large time behavior of \eqref{eq},\eqref{equ}, we define new weight function $\eta$ and new periodic solutions $\tilde{u}_{\pm},$

\begin{equation}
\eta:=\frac{{u}^{c}-u_{-}}{|u_-|},\quad \tilde{u}_{\pm}=\bar{u}_{\pm}-u_{\pm},
\end{equation}
then the proper ansatz $\bar{u}$ is constructed as

\begin{equation}\label{ansatz}
\begin{aligned}
%&\bar{u}^c=(1-\eta)\tilde{u}_{-}+\eta{u}^{0}+{u}^c,\ \bar{u}^r=(1-\eta)u^0+\eta\tilde{u}_{+}+{u}^r \\
\bar{u}=(1-\eta)\tilde{u}_-+\eta\tilde{u}_++\hat{u},
\end{aligned}
\end{equation}
which is periodic in the $x_i$ direction for $i=2,...,n.$

According to the definition of $\bar{u}$ in \eqref{ansatz}, we get
\begin{align}\label{pU}
&\partial_t\bar{u}=\partial_{t}\eta(\tilde{u}_+-\tilde{u}_-)+(1-\eta)\partial_{t}\tilde{u}_-+\eta\partial_{t}\tilde{u}_++\partial_{t}\hat{u}, \\
&\partial_{t}\tilde{u}_{\pm}=-\sum\limits_{i=1}^{n}\partial_{x_i}(f_i(\bar{u}_{\pm})-f_i(u_{\pm}))+\sum\limits_{i, j=1}^{n} a_{i j} (\tilde{u}_{\pm})_{x_{i} x_{j}}\nonumber .
%&\partial_{t} u^0=-\sum\limits_{i=1}^{n}\partial_{x_i}(f_i(u^{0}))+\sum\limits_{i, j=1}^{n} a_{i j} {u}^0_{x_{i} x_{j}}.\nonumber
\end{align}

Note that $\hat{u}$ is independent of $x_i$, $i=2,...,n$. By direct calculation, $\hat{u}$ satisfies
\begin{align}\label{BU}
{\partial_{t} \hat{u}+\sum\limits_{i=1}^{n}\partial_{x_i}(f_i(\hat{u}))=\sum\limits_{i, j=1}^{n} a_{i j} \hat{u}_{x_{i} x_{j}}}-N(u^c,u^r),
\end{align}
where $N(u^c,u^r)$ is the nonlinear interaction term which is also independent of $x_i$, $i=2,...,n$.
\begin{align}\label{N}
-N(u^c,u^r)\equiv \left(f_1^{\prime}\left(u^c+u^{r}\right)-f_1^{\prime}\left(u^{r}\right)\right) \partial_{x_1} u^{r}+f_1^{\prime}\left(u^c+u^{r}\right) \partial_{x_1} u^c+a_{11}\partial_{x_1}^2 u^r.
\end{align}
Combining \eqref{pU}-\eqref{N}, we have
\begin{equation}\label{baru}
\begin{aligned}
\partial_{t}\bar{u}+\sum\limits_{i=1}^{n}\partial_{x_i}(f_i(\bar{u}))=\sum\limits_{i, j=1}^{n} a_{i j} \bar{u}_{x_{i} x_{j}}+J,
\end{aligned}
\end{equation}
where
\begin{equation}\label{key}
\begin{aligned}
J=&\bigg\{\sum\limits_{i=1}^{n}\partial_{x_i}(f_i(\bar{u})-f_i(\hat{u}))-(1-\eta)\partial_{x_i}(f_i(\bar{u}_{-})-f_i(u_{-}))-\eta\partial_{x_i}(f_i(\bar{u}_{+})-f_i(u_{+}))\\
&-\sum_{i=1}^{n}a_{i1}\partial_{x_1}\eta(\tilde{u}_{+}-\tilde{u}_-)_{x_i}-\sum_{j=1}^{n}a_{1j}\partial_{x_{1}}\eta(\tilde{u}_+-\tilde{u}_-)_{x_j}-a_{11}\partial_{x_{1}}^2\eta(\tilde{u}_+-\tilde{u}_-)\\
&+\partial_{t}\eta(\tilde{u}_+-\tilde{u}_-)\bigg\}+\bigg\{-N(u^c,u^r)\bigg\}=J_1+J_2.
\end{aligned}
\end{equation}
The properties of $J$ can be obtained quickly.

\begin{Lemma}[\cite{Y}](The $L^p$ estimates for J)\label{J} For $\forall \epsilon>0,$  it holds that
	
	\begin{equation}\label{Jp}
	\begin{aligned}
	\|J\|_{L^p(\Omega)}\leq C_{\epsilon,p}(1+t)^{-\frac{1}{2}(1+\frac{1}{p}(1-\epsilon))},\quad \forall p\in[1,+\infty)\quad  (t\geq T_0).
	\end{aligned}
	\end{equation}
	
\end{Lemma}
\begin{proof}Note that $\|J\|_{L^p}\leq \|J_1\|_{L^p}+\|J_2\|_{L^p},$ then  we divided it into two parts:

	Step 1: After simple calculation, using lemmas \ref{rl}-\ref{cl}, we have
	\begin{align}
	J_1=&O\{(\bar{u}-\hat{u})\partial_{x_{1}}\hat{u}\}+O\{(1-\eta)(\bar{u}-\bar{u}_-)\partial_{x_i}\tilde{u}_-,\eta(\bar{u}-\bar{u}_+)\partial_{x_i}\tilde{u}_+\}\nonumber\\
	&+O\{|\partial_{x_1}\eta||(\tilde{u}_{\pm})|,|\partial_{x_1}\eta||(\tilde{u}_{\pm})_{x_i}|,|\partial_{x_{1}}^2\eta||\tilde{u}_{\pm}|\}=:J_{11}+J_{12}+J_{13}.\nonumber
	\end{align}
	And
	\begin{equation}\label{J11}	
	\begin{aligned}
	&\|J_{11}\|_{L^1(\Omega)}=\int_{\Omega}|(\bar{u}-\hat{u})\partial_{x_1}\hat{u}|dx
	\leq  \int_{\Omega} |(1-\eta)\tilde{u}_-+\eta\tilde{u}_+|\partial_{x_1}(u^c+u^r) dx  \\
	\leq &C\varepsilon e^{-\bar{c}t}\bigg(\int_{\Omega}e^{-\frac{c_0x_1^2}{1+t}}dx+\|\partial_{x_{1}} u^r\|_{L^1(\Omega)}\bigg)\leq C\varepsilon e^{-\bar{c}t},
	\end{aligned}
	\end{equation}	
	\begin{equation}\label{J12}
	\begin{aligned}
	&\|J_{12}\|_{L^1(\Omega)}=\int_{\Omega}|(1-\eta)(\bar{u}-\bar{u}_-)\partial_{x_i}\tilde{u}_-+\eta(\bar{u}-\bar{u}_+)\partial_{x_i}\tilde{u}_+|dx \\
	%\leq &\int_{\Omega}|((1-\eta)\eta(\tilde{u}_+-\tilde{u}_-)+(1-\eta)(\hat{u}-u_-))\partial_{x_i}\tilde{u}_-|dx\\
	%&+\int_{\Omega}|(\eta(1-\eta)(\tilde{u}_--\tilde{u}_+)+\eta(\hat{u}-u_+))\partial_{x_i}\tilde{u}_+| dx\\
	\leq &\int_{\Omega}|\eta(1-\eta)||(\tilde{u}_{\pm},(\tilde{u}_{\pm})_{x_i})|+|(1-\eta)(u^r-u_-),\eta(u^r-u_+)||(\tilde{u}_{\pm})_{x_i}|dx\\
	\leq &C\varepsilon e^{-\bar{c}t}\int_{\Omega}e^{-\frac{c_0x_1^2}{1+t}}dx
	+C_{\delta}\varepsilon e^{-\bar{c}t}(1+t)^{-1+\delta}\bigg(\int_{\mathbb{R}_{\pm}\times\mathbb{T}^{n-1}} \mathrm{e}^{-\delta\left|x-\lambda_{\pm} t\right|}dx\bigg)\\
	\leq &C\varepsilon e^{-\bar{c}t}.
	\end{aligned}
	\end{equation}
	Similarly, $\|J_{13}\|_{L^1(\Omega)}\leq C\varepsilon e^{-\bar{c}t}.$ Then we get $\|J_1\|_{L^1(\Omega)}\leq C\varepsilon e^{-\bar{c}t}.$
	
	Besides that, \eqref{pe-de} tell us that $\|J_1(t)\|_{L^{\infty}(\Omega)}\leq C\varepsilon e^{-\bar{c}t}.$ 	
	As for $p\in (1,\infty),$ it yields that
	\begin{equation}\label{J1}
	\|J_1(t)\|_{L^p(\Omega)}\leq \|J_1(t)\|_{L^{\infty}(\Omega)}^{1-\frac{1}{p}}\|J_1(t)\|_{L^{1}(\Omega)}^{\frac{1}{p}}\leq C\varepsilon e^{-\bar{c}t}.
	\end{equation}
	Step 2: Note that $J_2=N(u^c,u^r),$ and this term is only related to $x_1$, is unrelated to $x'.$ Hence, the $L^p$ estimates of this term on $\Omega$ is similar as the term $F(U,U^r)$ in \cite{Y}. We omit the details(see Proposition 3.1 in \cite{Y}):
	
	\begin{equation}\label{J2}
	\begin{aligned}
	\|J_2(t)\|_{L^p(\Omega)}\leq C_{\epsilon,p}(1+t)^{-\frac{1}{2}(1+\frac{1}{p}(1-\epsilon))},\quad (t\geq T_0).
	\end{aligned}
	\end{equation}
	By \eqref{J1}-\eqref{J2}, we get \eqref{Jp}.
	
\end{proof}
\subsection{Main results}

Now we present our main theorem

\begin{Thm}\label{mt}Assume that the flux function $f$ satisfies \eqref{f''}, the far field states $u_-<u_+,$ and the periodic perturbation $V_0(x)\in H^{[\frac{n}{2}]+2}(\mathbb{T}^{n})$ satisfies \eqref{cd-zeroav}. Then there exists a unique global smooth solution $u$ of \eqref{eq},\eqref{equ} satisfying
	\begin{equation}\label{st}
	\begin{aligned}
	\|u(x,t)-\hat{u}(x_1,t)\|_{L^{\infty}(\mathbb{R}^{n})}\leq C_{\epsilon}(1+t)^{-\frac{1}{2}+\epsilon},
	\end{aligned}
	\end{equation}
	for any $\epsilon>0,$ $C_{\epsilon}>0$  depending on $\epsilon$, and $\hat{u}(x_1,t)$ is defined in \eqref{hatu}.
\end{Thm}

\begin{Remark}The stability results \eqref{st} is proved by the $L^p$ energy estimates.  Since the original perturbation $u(x,t)-\hat{u}(x_1,t)$ is not integrable in $L^p$ space, to overcome this, a new ansatz $\bar{u}$\eqref{ansatz} is constructed which is periodic in $x_i, i=2,...,n$ direction. On the one hand, for $k\geq 0,$
	
	\begin{equation}\label{exd}
	\|\partial_{x}^{k}(\bar{u}(x,t)-\hat{u}(x_1,t))\|_{L^{\infty}(\mathbb{R}^{n})}\leq C\varepsilon e^{-\bar{c}t},
	\end{equation}
	on the other hand, using the $L^p$ estimates for the new function $\phi=u(x,t)-\bar{u}(x,t)$, we could finally get \eqref{st} .
\end{Remark}
\begin{Remark}
	We obtain the same decay rate as in Theorem 1.4 in \cite{Y}, but our initial perturbations are not integrable, which is the key condition in \cite{Y}. Without this condition, the decay rate is only $(1+t)^{-1/4+\epsilon}.$
\end{Remark}

\section{Reformulation of the problem}

\subsection{The perturbation function $\phi(x,t)$}

Setting
\begin{align*}
u(x,t)\equiv \bar{u}(x,t)+\phi(x,t).
\end{align*}

We obtain the perturbation equation by (\ref{eq}), \eqref{baru},
\begin{align}\label{aape}
\begin{cases}
\partial_{t} \phi+\sum\limits_{i=1}^{n}\partial_{x_i}(f_i(\bar{u}+\phi)-f_i(\bar{u}))=\sum\limits_{i, j=1}^{n} a_{i j} \phi_{x_ix_j}-J,\\
\phi(x,0)=\phi_{0}(x)\equiv u_{0}(x)-\bar{u}(x,0)=0.
\end{cases}
\end{align}
Now we list some energy results for $\phi.$

\begin{Lemma}\label{Li}[$L^{\infty}$-boundness]The unique solution $\phi(x,t)$ of the Cauchy problem \eqref{aape} satisfies
	\begin{equation}
	\sup_{t\in[0,+\infty),x\in\Omega}|\phi(x,t)|\leq C.
	\end{equation}
\end{Lemma}
\begin{Lemma}\label{H1}[$H^{1}$-boundness]The unique solution $\phi(x,t)$ of the  Cauchy problem \eqref{aape} satisfies
	\begin{equation}
	\begin{aligned}
	\|\phi(t)\|^2_{H^1(\Omega)}+\int_{0}^{+\infty}\bigg(Q_2(\tau)+\|\nabla\phi(\tau)\|^2_{H^1(\Omega)}\bigg)d\tau\leq C(\phi_0),\quad (t\geq 0),
	\end{aligned}
	\end{equation}	
	where $Q_2=Q_2(t)$ is given by
	\begin{equation}
	\begin{aligned}
	Q_2=\bigg(\int_{\bar{u}+\phi> 0, \bar{u}>0}|\phi|^{2} \partial_{x_1} \hat{u} \mathrm{d} x+\int_{\bar{u}+\phi> 0, \bar{u}\leq 0}\left(\bar{u}+\phi\right)^2\partial_{x_1} \hat{u}dx+\int_{\bar{u}+\phi\leq 0, \bar{u}> 0}\bar{u}^2\partial_{x_{1}}\hat{u}dx\bigg).
	\end{aligned}
	\end{equation}

\end{Lemma}
The $L^{\infty}$ boundness of $\phi$ can be obtained by the maximum principle. The $H^1$ boundness of $\phi$ can be verified easily, here we omit the proof of lemma \ref{Li}-\ref{H1}.

\begin{Lemma}[$L^1$-estimates]\label{L1}For $\forall \epsilon>0,$ the unique solution $\phi(x,t)$ of the Cauchy problem \eqref{aape} satisfies
	\begin{equation}
	\|\phi\|_{L^1(\Omega)}\leq C(1+t)^{\epsilon}.
	\end{equation}
\end{Lemma}

\begin{proof} Similar as \cite{HY}, given $\sigma>0$ and let $S_{\sigma}(\eta)$ be a $C^2$ convex approximation to the function $|\eta|,$ e.g.,
	\begin{equation}\label{s1}
	S_{\sigma}(\eta)=\left\{
	\begin{aligned}
	&-\eta,\quad \eta\leq -\sigma;\\
	&-\frac{\eta^4}{8\sigma^3}+\frac{3\eta^2}{4\sigma}+\frac{3\sigma}{8},\quad -\sigma<\eta\leq \sigma;\\
	&\eta,\quad \eta>\sigma.
	\end{aligned}
	\right.
	\end{equation}
	
	Multiplying $S_{\sigma}^{'}(\phi)$ on both sides of \eqref{aape} yields that
	\begin{equation}\label{s2}
	\begin{aligned}
	&\partial_{t}S_{\sigma}(\phi)+\sum_{i,j=1}^{n}a_{ij}S_{\sigma}''(\phi)\phi_{x_i}\phi_{x_j}+\int_{0}^{\phi}S_{\sigma}''(\eta)(f_{1}'(\bar{u}+\eta)-f_{1}'(\bar{u}))d\eta\partial_{x_{1}}\hat{u}\\
	=&-J\cdot S_{\sigma}'(\phi)+\sum_{i=1}^{n}\partial_{x_{i}}(\cdots)-\sum_{i=1}^{n}\int_{0}^{\phi}S_{\sigma}''(\eta)(f_{i}'(\bar{u}+\eta)-f_{i}'(\bar{u}))d\eta\partial_{x_{i}}(\bar{u}-\hat{u}),
	\end{aligned}
	\end{equation}
	where
	\begin{equation}\label{s3}
	\begin{aligned}
	\{\cdots\}=\sum_{j=1}^{n}(a_{ij}S_{\sigma}(\phi))_{x_j}-S_{\sigma}'(\phi)(f_{i}(\bar{u}+\phi)-f_{i}(\bar{u}))+\int_{0}^{\phi}S_{\sigma}''(\eta)(f_{i}'(\bar{u}+\eta)-f_{i}'(\bar{u}))d\eta.
	\end{aligned}
	\end{equation}
	Since $S_{\sigma}^{''}\geq 0, f_1^{''}\geq 0,\partial_1\hat{u}>0$ and $|\phi|\leq S_{\sigma}(\phi),$ from \eqref{exd} and
	\begin{equation}
	\begin{aligned}
	\sum_{i,j=1}^{n}a_{ij}\phi_{x_i}\phi_{x_j}\geq b|\nabla\phi|^2.
	\end{aligned}
	\end{equation}

	Integrating \eqref{s2} over $\Omega$, we have
	\begin{equation}
	\begin{aligned}
	&\frac{d}{dt}\int_{\Omega}S_{\sigma}(\phi)dx\leq Ce^{-c t}\int_{\Omega}|\int_{0}^{\phi}S_{\sigma}''(\eta)|\eta|d\eta|dx+\|J(t)\|_{L^1(\Omega)}\\
	%\leq & Ce^{-ct}\|\phi(t)\|_{L^1(\Omega)}+C(1+t)^{-(1-\epsilon)}\\
	\leq & Ce^{-ct}\int_{\Omega}S_{\sigma}(\phi)dx+C(1+t)^{-(1-\epsilon)}.
	\end{aligned}
	\end{equation}
	By using the Gronwall inequality and let $\sigma\rightarrow 0+,$ one has that

	\begin{equation}
	\begin{aligned}
	\|\phi(t)\|_{L^1(\Omega)}\leq C(1+t)^{\epsilon}.
	\end{aligned}
	\end{equation}
	
\end{proof}

\subsection{The decomposition for $\phi$}
To get our stability results \eqref{st}, we need to decompose the solution $\phi(x,t)$ into the principal and transversal parts. We set $\int_{\mathbb{T}^{n-1}}1dx'=1,$ then we can define the following decomposition $\mathbf{D}_0$ and $\mathbf{D}_{\neq},$
\begin{equation}\label{def-decom}
\begin{aligned}
\mathbf{D}_{0}f:= \mr{f}:=\int_{\mathbb{T}^{n-1}}f dx',\ \mathbf{D}_{\neq}f:=\ac{f}:=f-\mr{f},
\end{aligned}
\end{equation}
for an arbitrary function $f$ which is integrable on $\mathbb{T}^{n-1}$.  With simple analysis, the following propositions of $\mathbf{D}_0$ and $\mathbf{D}_{\neq}$ hold for an arbitrary function $f$ which is integrable on $\mathbb{T}^{n-1}$ .

\begin{Lemma}\label{lemma-decom}
	For the projections $\Do$ and $\Dn$ defined in \eqref{def-decom}, the following holds,
	
	i) $\Do\Dn f=\Dn\Do f=0$.
	
	ii) For any non-linear function $F$, one has
	\begin{align}
	\Do F(U)-F(\Do U)=O(1) (\Dn U)^2,
	\end{align}

	iii) $\|f\|^2=\|\Do f\|^2+\|\Dn f\|^2.$
\end{Lemma}

Applying $\Do$ to \eqref{aape} , we decompose the perturbation $\phi$ into the zero mode $\mr{\phi}$ and the non-zero mode $\ac{\phi}$ ($\phi=\mr{\phi}+\ac{\phi}$),
\begin{equation}\label{zm}
\partial_{t}\mr{\phi}+\partial_{x_1}\bigg(\Do(f_1(\bar{u}+\phi)-f_1(\bar{u}))\bigg)=a_{11}\partial_{x_1}^2\mr{\phi}-\mr{J},
\end{equation}
\begin{equation}\label{nzm}
\pt\ac{\phi}+\sum\limits_{i=1}^{n}\partial_{x_i}\bigg\{f_i(\bar{u}+\phi)-f_i(\bar{u})-\Do\big(f_i(\bar{u}+\phi)-f_i(\bar{u})\big)\bigg\}=\sum\limits_{i, j=1}^{n} a_{i j} \ac{\phi}_{x_ix_j}-\ac{J}.
\end{equation}
Here the expression of $\mr{J},\ac{J}$ are

\begin{equation}\label{mr-ac J}
\begin{aligned}
&\mr{J}=\mr{J}_1+\mr{J}_2=\Do(J_1)-N(u^c,u^r),\\
&\ac{J}=J-\mr{J}=J_1-\mr{J}_1,
\end{aligned}
\end{equation}
respectively(refer to \eqref{key}).

Now we list some $L^p$ energy estimates for $\mr{\phi},\ac{\phi}$, using them and combining them with a standard continuity argument, we could finally prove theorem \ref{mt}.

\begin{Proposition}[local in time existence]\label{let}
	For $T>0$ suitably small, there exists a unique smooth solution $\phi(x,t)$ for   the initial value problem (\ref{aape}) on the time interval $\left[0, T\right].$
\end{Proposition}

\begin{Remark}
	We omit the proof of proposition \ref{let} since it is very standard. because $\phi=\mr{\phi}+\ac{\phi},$ and the equations for $\mr{\phi},\ac{\phi}$ are also uniformly parabolic,see \eqref{zm}-\eqref{nzm}. The local smooth solution for $\mr{\phi},\ac{\phi}$ could also be obtained.
\end{Remark}

\begin{Proposition}[a priori estimate for the non-zero mode $\ac{\phi}$]\label{pet1}
	If the solution $\ac{\phi}(x,t)$ is the local smooth solution obtained in Proposition \ref{let},  then for $t\in[0,T],$ we have
	\begin{equation}\label{p3.2}
	\begin{aligned}
	&\|\ac{\phi}(t)\|_{L^p({\Omega})}\leq C e^{-\bar{c}t},\quad  \forall p\in [2,+\infty),\\
	&\|\nabla\ac{\phi}(t)\|_{L^p(\Omega)}\leq C e^{-\bar{c}t}, \quad  \forall p\in [2,+\infty).
	\end{aligned}
	\end{equation}	
\end{Proposition}

\begin{Proposition}[a priori estimate for the zero mode $\mr{\phi}$]\label{pet2}
	If the solution $\mr{\phi}(x,t)$ is the local smooth solution obtained in Proposition \ref{let},  then for $t\in[0,T],\forall \epsilon>0,$ we have
	\begin{equation}\label{p3.3}
	\begin{aligned}
	&\|\mr{\phi}(t)\|_{L^p({\mathbb{R}})}\leq C_{p,\epsilon}(1+t)^{-\frac{1}{2}(1-\frac{1}{p})+\epsilon},\quad  \forall p\in [1,+\infty),\\
	&\|\partial_{x_1}\mr{\phi}\|_{L^p(\mathbb{R})}\leq C_{p,\epsilon}(1+t)^{-\frac{1}{2}(1-\frac{1}{p})+\epsilon},\ \forall p\in[2,+\infty),
	\end{aligned}
	\end{equation}
	where $C_{p,\epsilon}>0$  depending on $p,\epsilon.$
\end{Proposition}

In the next section, we mainly prove proposition \ref{pet1}-\ref{pet2}.  Then we finally get theorem \ref{mt} by using those two propositions.

\section{A priori estimates}
This section is devoted to proving proposition \ref{pet1}-\ref{pet2}. Before doing this, we need to prepare some assumptions for $\ac{\phi},\mr{\phi}$.

From  lemmas \ref{Li}-\ref{L1}, there exists $T>0$ such that for $t\in[0,T],$ we have
\begin{equation}\label{nu}
\begin{aligned}
&\sup_{t\in[0,T]}\|\mr{\phi}\|_{L^{\infty}(\mathbb{R})}\leq \nu,\quad
\sup_{t\in[0,T]}\|\mr{\phi}\|_{H^1(\mathbb{R})}\leq C,\quad \sup_{t\in[0,T]}\|\mr{\phi}\|_{L^1(\mathbb{R})}\leq C(1+t)^{\epsilon},\\
&\sup_{t\in[0,T]}\|\ac{\phi}\|_{L^{\infty}(\Omega)}\leq \nu,\quad
\sup_{t\in[0,T]}\|\ac{\phi}\|_{H^1(\Omega)}\leq C,\quad \sup_{t\in[0,T]}\|\ac{\phi}\|_{L^1(\Omega)}\leq C(1+t)^{\epsilon},\\
&\sup_{t\in[0,T]}\|(\mr{\phi},\nabla\mr{\phi})\|_{L^p(\mathbb{R})}\leq C,\sup_{t\in[0,T]}\|(\ac{\phi},\nabla{\ac{\phi}})\|_{L^p(\Omega)}\leq C, \ \forall p\in(2,+\infty),
\end{aligned}
\end{equation}
where $1>\nu>0$ is suitably small. This is because for $m\geq 0,1\leq p\leq+ \infty,$
\begin{equation}
\begin{aligned}
&\|\nabla^{m}\mr{\phi}\|_{L^{p}(\mathbb{R})}\leq \parallel\|\nabla^{m}\mr{\phi}\|_{L^1(\mathbb{T}^{n-1})}\parallel_{L^{p}(\mathbb{R})}\leq \|\nabla^{m}\phi\|_{L^p(\Omega)},\\
&\|\nabla^{m}\ac{\phi}\|_{L^{p}(\Omega)}\leq \|\nabla^{m}(\phi-\mr{\phi})\|_{L^{p}(\Omega)}\leq 2\|\nabla^m\phi\|_{L^p(\Omega)}.
\end{aligned}
\end{equation}

\subsection{Proof of Proposition \ref{pet1}}
Now we start to give the $L^p$ estimates for the non-zero mode $\ac{\phi},$ from \eqref{aape},\eqref{nzm}, the initial problem of $\ac{\phi}$ is following
\begin{equation}\label{nzmi}
\left\{
\begin{aligned}
&\pt\ac{\phi}+\sum\limits_{i=1}^{n}\partial_{x_i}\bigg\{f_i(\bar{u}+\phi)-f_i(\bar{u})-\Do\big(f_i(\bar{u}+\phi)-f_i(\bar{u})\big)\bigg\}=\sum\limits_{i, j=1}^{n} a_{i j} \ac{\phi}_{x_ix_j}-\ac{J},\\
&\ac{\phi}(x,0)=0.
\end{aligned}
\right.
\end{equation}

\begin{Lemma}\label{GN}(\cite{HY})Assume that $w\in L^q(\Omega)$ with $\nabla^m w\in L^r(\Omega)$, where $1\leq q,r\leq +\infty$ and $m\geq 1,$ and $w$ is periodic in the $x_i$ direction for $i=2,\cdots,n.$ Then there exists a decomposition $w(x)=\sum\limits_{k=0}^{n-1}w^{(k)}(x)$ such that each $w^{(k)}$ satisfies the $k+1$-dimensional $GN$ inequality, i.e.,
	\begin{equation}\label{GN1}
	\begin{aligned}
	\|\nabla^{j}w^{(k)}\|_{L^p(\Omega)}\leq C\|\nabla^m w\|_{L^r(\Omega)}^{\theta_k}\|w\|_{L^q(\Omega)}^{1-\theta_k},
	\end{aligned}
	\end{equation}	
	for any $0\leq j<m$ and $1\leq p\leq +\infty$ satisfying $\frac{1}{p}=\frac{j}{k+1}+(\frac{1}{r}-\frac{m}{k+1})\theta_k+\frac{1}{q}(1-\theta_k)$ and $\frac{j}{m}\leq \theta_k\leq 1.$ Hence, it holds that
	\begin{equation}\label{GN2}
	\begin{aligned}
	\|\nabla^j w\|_{L^p(\Omega)}\leq C\sum_{k=0}^{n-1}\|\nabla^m w\|_{L^r(\Omega)}^{\theta_k}\|w\|_{L^q(\Omega)}^{1-\theta_k},\ (t\geq 0),
	\end{aligned}
	\end{equation}
	where the constant $C>0$ is independent of $u$. Moreover, we get that for any $2\leq p<\infty$ and $1\leq q\leq p,$ it holds that	
	\begin{equation}\label{GN3}
	\begin{aligned}
	\|w\|_{L^p(\Omega)}\leq C\sum_{k=0}^{n-1}\|\nabla(|w|^{\frac{p}{2}})\|_{L^2(\Omega)}^{\frac{2\gamma_k}{1+\gamma_kp}}\|w\|_{L^q(\Omega)}^{\frac{1}{1+\gamma_kp}},
	\end{aligned}
	\end{equation}
	where $\gamma_k=\frac{k+1}{2}(\frac{1}{q}-\frac{1}{p})$ and the constant $C=C(p,q,n)>0$ is independent of $u$.

\end{Lemma}
\begin{proof}Please refer to \cite{HY}.
\end{proof}

\begin{Lemma}[The basic $L^p$ estimate for $\ac{\phi}, 2\leq p<+\infty$]\label{L4.1}
	
\begin{equation}\label{l4.01}
\begin{aligned}
\frac{d}{dt}\|\ac{\phi}\|_{L^p}^p+b\|\nabla|\ac{\phi}|^{\frac{p}{2}}\|_{L^2}^2\leq C\varepsilon e^{-\bar{c}t\cdot p}+(C\varepsilon e^{-\bar{c}t}+\nu)\|\ac{\phi}\|_{L^p}^p.
\end{aligned}
\end{equation}		
\end{Lemma}
\begin{proof}For $p\in[2,+\infty)$, multiplying \eqref{nzmi} by $|\ac{\phi}|^{p-2}\ac{\phi}$,  it yields that
\begin{equation}\label{l4.11}
\begin{aligned}
&\frac{1}{p}\pt|\ac{\phi}|^p+(p-1)\sum\limits_{i, j=1}^{n} a_{i j}|\ac{\phi}|^{p-2}\ac{\phi}_{x_i}\ac{\phi}_{x_j}+\sum_{i=1}^{n}\pxi(\cdots)\\
=&\sum\limits_{i=1}^{n}\bigg\{f_i(\bar{u}+\phi)-f_i(\bar{u})-\Do\big(f_i(\bar{u}+\phi)-f_i(\bar{u})\big)\bigg\}\partial_{x_i}(|\ac{\phi}|^{p-2}\ac{\phi})-\ac{J}|\ac{\phi}|^{p-2}\ac{\phi},
\end{aligned}
\end{equation}
where $(\cdots)$ equal to
\begin{equation}\label{l4.12}
\begin{aligned}
-\sum\limits_{j=1}^{n}a_{ij}\ac{\phi}_{x_j}|\ac{\phi}|^{p-2}\ac{\phi}+\bigg\{f_i(\bar{u}+\phi)-f_i(\bar{u})-\Do\big(f_i(\bar{u}+\phi)-f_i(\bar{u})\big)\bigg\}|\ac{\phi}|^{p-2}\ac{\phi}.
\end{aligned}
\end{equation}
Using \eqref{J1}, \eqref{mr-ac J}, one has
\begin{equation}\label{l4.13}
\begin{aligned}
\int_{\Omega}|\ac{J}||\ac{\phi}|^{p-2}\ac{\phi}dx\leq \|\ac{J}\|_{L^p}\|\ac{\phi}\|_{L^p}^{p-1}\leq C\|J_1\|_{L^p}\|\ac{\phi}\|_{L^p}^{p-1}\leq C\varepsilon e^{-\bar{c}t\cdot p}+\varepsilon e^{-\bar{c}t}\|\ac{\phi}\|_{L^p(\mathbb{R})}^p.
\end{aligned}
\end{equation}
As for the first term on the right hand-side of \eqref{l4.11}, remember $\phi=\mr{\phi}+\ac{\phi},$ from lemma \ref{lemma-decom},
\begin{equation}\label{l4.14}
\begin{aligned}
&\bigg\{f_i(\bar{u}+\phi)-f_i(\bar{u})-\Do\big(f_i(\bar{u}+\phi)-f_i(\bar{u})\big)\bigg\}\\
=&\bigg(f_i'(\bar{u}){\phi}-\Do(f_i'(\bar{u})\phi\big)\bigg)\\
&+\big[f_i(\bar{u}+\phi)-f_i(\bar{u})-f_i'(\bar{u})\phi-\Do\big(f_i(\bar{u}+\phi)-f_i(\bar{u})-f_i'(\bar{u})\phi\big)\big]\\
=&\bigg(f_i'(\bar{u})\ac{\phi}\bigg)+\bigg((f_i'(\bar{u})-f_i'(\mr{\bar{u}}))\mr{\phi}+\Do(f_i'(\mr{\bar{u}})\mr{\phi}-f_i'(\bar{u})\phi)\bigg)\\
&+f_i^{''}(\bar{u}+\theta\phi)(\phi^2-\mr{\phi}^2)+\bigg(f_i^{''}(\bar{u}+\theta\phi)-\Do(f_i^{''}(\bar{u}+\theta\phi))\bigg)\mr{\phi}^2
-\Do(f_i^{''}(\bar{u}+\theta\phi)\ac{\phi}^2)\\
=&\bigg(f_i'(\bar{u})\ac{\phi}\bigg)+O(1)\bigg(\ac{\bar{u}}\mr{\phi}+\ac{\bar{u}}\ac{\phi}+\ac{\phi}^2+\ac{\bar{u}}\mr{\phi}^2+\ac{\phi}\mr{\phi}^2\bigg),
\end{aligned}
\end{equation}
then we have
\begin{equation}\label{l4.16}
\begin{aligned}
I_1=&\int_{\Omega}\bigg(f_i'(\bar{u})\ac{\phi}\bigg)\partial_{x_i}(|\ac{\phi}|^{p-2}\ac{\phi})dx\\
=&\int_{\Omega}\partial_{x_i}\bigg(\frac{p-1}{p}f_i'(\bar{u})|\ac{\phi}|^{p}\bigg)-\frac{p-1}{p}f^{''}_i(\bar{u})|\ac{\phi}|^{p}\partial_{x_{i}}\bar{u}dx.
\end{aligned}
\end{equation}
For $i=1,$
\begin{equation}\label{l4.17}
\begin{aligned}
&-\int_{\Omega}f^{''}_1(\bar{u})|\ac{\phi}|^{p}\partial_{x_{1}}\hat{u}dx<0, \ \int_{\Omega}f^{''}_1(\bar{u})|\ac{\phi}|^{p}|\partial_{x_{1}}(\bar{u}-\hat{u})|dx\leq C\varepsilon e^{-\bar{c}t}\|\ac{\phi}\|_{L^p}^p.
\end{aligned}
\end{equation}
For $i\neq 1,$ \eqref{pe-de}, \eqref{nu} and $\|\ac{\bar{u}}\|_{L^{\infty}}\leq C\varepsilon e^{-\bar{c}t}$ implies that
\begin{equation}\label{l4.18}
\begin{aligned}
\int_{\Omega}|f^{''}_i(\bar{u})||\ac{\phi}|^{p}|\partial_{x_{i}}\bar{u}|dx\leq C\varepsilon e^{-\bar{c}t}\|\ac{\phi}\|_{L^p}^p.
\end{aligned}
\end{equation}
Thus,
\begin{equation}\label{l4.19}
\begin{aligned}
I_2=&\int_{\Omega}(\ac{\bar{u}}\mr{\phi}+\ac{\bar{u}}\ac{\phi}+\ac{\phi}^2+\ac{\bar{u}}\mr{\phi}^2+\ac{\phi}\mr{\phi}^2)|\ac{\phi}|^{p-2}\partial_{x_i}\ac{\phi}dx\\
\leq &\nu\big(\|\nabla|\ac{\phi}|^{\frac{p}{2}}\|_{L^2}^2+\|\ac{\phi}\|_{L^p}^p)+C\varepsilon e^{-\bar{c}t\cdot p}\|\mr{\phi}\|_{L^p(\mathbb{R})}^p+C\varepsilon e^{-\bar{c}t}\|\ac{\phi}\|_{L^p}^p.
\end{aligned}
\end{equation}
Combining \eqref{l4.13}-\eqref{l4.19}  and integrating \eqref{l4.11} over $\Omega$, it yields that

\begin{equation}
\begin{aligned}
\frac{d}{dt}\|\ac{\phi}\|_{L^p}^p+b\|\nabla|\ac{\phi}|^{\frac{p}{2}}\|_{L^2}^2\leq C\varepsilon e^{-\bar{c}t\cdot p}+(C\varepsilon e^{-\bar{c}t}+\nu)\|\ac{\phi}\|_{L^p}^p.
\end{aligned}
\end{equation}

\end{proof}

\begin{Lemma}[Time decay estimate for $\ac{\phi}, 2\leq p<+\infty$]\label{L4.2}
	
	\begin{equation}\label{l4.21}
	\begin{aligned}
	&\|\ac{\phi}(t)\|_{L^p({\Omega})}\leq C e^{-\bar{c}t}, \quad \forall p\in [2,+\infty).
	\end{aligned}
	\end{equation}

\end{Lemma}

\begin{proof}
$Step$ 1:  From lemma \ref{L4.1}, when $p=2,$
\begin{equation}\label{14.22}
\begin{aligned}
\frac{d}{dt}\|\ac{\phi}\|_{L^2}^2+b\|\nabla\ac{\phi}\|_{L^2}^2\leq C\varepsilon e^{-\bar{c}t\cdot p}+(C\varepsilon e^{-\bar{c}t}+\nu)\|\ac{\phi}\|_{L^2}^2.
\end{aligned}
\end{equation}
Because $\int_{\Omega}\ac{\phi}dx=0,$ the poincare-inequality yields that
\begin{equation}\label{l4.23}
\begin{aligned}
\|\ac{\phi}\|_{L^2}^2\leq C\|\nabla\ac{\phi}\|_{L^2}^2,
\end{aligned}
\end{equation}
\begin{equation}\label{l4.24}
\begin{aligned}
\frac{d}{dt}\|\ac{\phi}\|_{L^2}^2+\|\ac{\phi}\|_{L^2}^2+\|\nabla\ac{\phi}\|_{L^2}^2
\leq C\varepsilon e^{-\bar{c}t}.
\end{aligned}
\end{equation}
Then we could get $\eqref{l4.21}$ for $p=2.$

$Step$ 2: When $p\in(2,+\infty),$ from \eqref{l4.01},

\begin{equation}\label{l4.25}
\begin{aligned}
\frac{d}{dt}\|\ac{\phi}\|_{L^p}^p+b\|\nabla|\ac{\phi}|^{\frac{p}{2}}\|_{L^2}^2\leq C\varepsilon e^{-\bar{c}t\cdot p}+C\|\ac{\phi}\|_{L^p}^p.
\end{aligned}
\end{equation}
Making use of GN-inequality \eqref{GN3} in lemma \ref{GN},

\begin{equation}\label{l4.27}
\begin{aligned}
\|\ac{\phi}\|_{L^p(\Omega)}\leq C\sum_{k=0}^{n-1}\|\nabla(|\ac{\phi}|^{\frac{p}{2}})\|_{L^2(\Omega)}^{\frac{2\gamma_k}{1+\gamma_kp}}\|\ac{\phi}\|_{L^2(\Omega)}^{\frac{1}{1+\gamma_kp}},
\end{aligned}
\end{equation}
\begin{equation}\label{l4.28}
\begin{aligned}
\frac{d}{dt}\|\ac{\phi}\|_{L^p}^p+b\|\nabla|\ac{\phi}|^{\frac{p}{2}}\|_{L^2}^2\leq & C\|\nabla(|\ac{\phi}|^{\frac{p}{2}})\|_{L^2(\Omega)}^{\frac{2\gamma_kp}{1+\gamma_kp}}\|\ac{\phi}\|_{L^2(\Omega)}^{\frac{p}{1+\gamma_kp}}+Ce^{-\bar{c}t\cdot p}\\
\leq &\nu \|\nabla|\ac{\phi}|^{\frac{p}{2}}\|_{L^2}^2+C\|\ac{\phi}\|_{L^2}^p+Ce^{-\bar{c}t\cdot p}\\
\leq &\nu \|\nabla|\ac{\phi}|^{\frac{p}{2}}\|_{L^2}^2+Ce^{-\bar{c}t\cdot p}.
\end{aligned}
\end{equation}
This implies that
\begin{equation}\label{l4.29}
\frac{d}{dt}\|\ac{\phi}\|_{L^p}^p+\|\ac{\phi}\|_{L^p}^p+\|\nabla|\ac{\phi}|^{\frac{p}{2}}\|_{L^2}^2\leq Ce^{-\bar{c}t\cdot p},
\end{equation}
we could get $\eqref{l4.21}$ for $p>2$.
\end{proof}

\begin{Remark}
	For $p\in(1,2),$ the interpolation inequality implies that
	\begin{equation}
	\|\ac{\phi}\|_{L^p}\leq \|\ac{\phi}\|_{L^1}^{\frac{2-p}{p}}\|\ac{\phi}\|_{L^2}^{\frac{2p-2}{p}}\leq Ce^{-\bar{c}t}\bigg\{(1+t)^{\epsilon}e^{\bar{c}t}\bigg\}^{\frac{2}{p}-1}.
	\end{equation}

\end{Remark}

\begin{Lemma}[Time decay estimates for $\nabla\ac{\phi}, 2\leq p<+\infty$]\label{l4.3}
	\begin{equation}\label{l4.33}
	\begin{aligned}
	\|\nabla\ac{\phi}(t)\|_{L^p({\Omega})}\leq C e^{-\bar{c}t}, \quad \forall p\in[2,+\infty).
	\end{aligned}
	\end{equation}	
\end{Lemma}
\begin{proof}Order $\mr{\phi}_k:=\pxk\mr{\phi},\ac{\phi}_k:=\pxk\ac{\phi},$ taking  the derivative on \eqref{nzmi} with respect to $x_k$. We have  for $k=1,$
\begin{equation}\label{l4.34}
\begin{aligned}
&\pt\ac{\phi}_1+\sum\limits_{i=1}^{n}\partial_{x_i}\bigg(f_i^{'}(\bar{u})\ac{\phi}_1-\Do(f_i^{'}(\bar{u})\ac{\phi}_1)\bigg)+
\partial_{x_i}\bigg\{\big(f_i^{'}(\bar{u}+\phi)-f_i^{'}(\bar{u})\big)(\partial_1\bar{u}+\ac{\phi}_1)\bigg\}\\
&-\sum\limits_{i=1}^{n}\partial_{x_i}\bigg\{\Do\bigg[\big(f_i^{'}(\bar{u}+\phi)-f_i^{'}(\bar{u})\big)(\partial_1\bar{u}+\ac{\phi}_1)\bigg]\bigg\}\\
&+\sum_{i=1}^{n}\partial_{x_{i}}\bigg\{f_i^{'}(\bar{u}+\phi)\mr{\phi}_1-\Do(f_i^{'}(\bar{u}+\phi))\mr{\phi}_1\bigg\}\\
=&\sum\limits_{i, j=1}^{n} a_{i j} \ac{\phi}_{1x_ix_j}-\partial_1\ac{J},
\end{aligned}
\end{equation}
for $2\leq k\leq n,$
\begin{equation}\label{l4.35}
\begin{aligned}
&\pt\ac{\phi}_k+\sum\limits_{i=1}^{n}\partial_{x_i}\bigg(f_i^{'}(\bar{u})\ac{\phi}_k\bigg)+
\partial_{x_i}\bigg\{\big(f_i^{'}(\bar{u}+\phi)-f_i^{'}(\bar{u})\big)(\partial_k\bar{u}+\ac{\phi}_k+\mr{\phi}_k)\bigg\}\\
=&\sum\limits_{i, j=1}^{n} a_{i j} \ac{\phi}_{kx_ix_j}-\partial_k\ac{J}.
\end{aligned}
\end{equation}
We only estimate the former one, since this two case are similar and the later one is easier. Multiplying \eqref{l4.34} by $|\ac{\phi}_1|^{p-2}\ac{\phi}_1$,  it yields that
\begin{equation}\label{l4.36}
\begin{aligned}
&\frac{1}{p}\pt|\ac{\phi}_1|^p+(p-1)\sum\limits_{i, j=1}^{n} a_{i j}|\ac{\phi}_1|^{p-2}\ac{\phi}_{1x_i}\ac{\phi}_{1x_j}+\sum_{i=1}^{n}\pxi(\cdots)\\
=&G-\partial_1\ac{J}|\ac{\phi}_1|^{p-2}\ac{\phi}_1,
\end{aligned}
\end{equation}
where the term $G$ can be divided into following situations,
\begin{equation}\label{G}
\begin{aligned}
G_1=&\bigg(f_i'(\bar{u})\ac{\phi}_1\bigg)\partial_{x_i}(|\ac{\phi}_1|^{p-2}\ac{\phi}_1),\\
G_2=&(\ac{\bar{u}}\mr{\phi}_1+\ac{\bar{u}}\ac{\phi}_1+\ac{\phi}\mr{\phi}_1+\mr{\phi}\ac{\phi}_1+\ac{\phi}\ac{\phi}_1)\partial_{x_i}(|\ac{\phi}_1|^{p-2}\ac{\phi}_1),\\
G_3=&\bigg\{(\ac{\phi}+\mr{\phi}+\mr{\phi}\ac{\phi})\partial_1\ac{\bar{u}}+(\ac{\phi}+\mr{\phi}\ac{\bar{u}}+\mr{\phi}\ac{\phi})\partial_1\mr{\bar{u}}\bigg\}\partial_{x_i}(|\ac{\phi}_1|^{p-2}\ac{\phi}_1).
\end{aligned}
\end{equation}
For $G_1$, the estimate is the same as $I_1$\eqref{l4.16}. For $G_2$, the holder inequality $\frac{1}{2p}+\frac{1}{2p}+\frac{p-2}{2p}+\frac{1}{2}=1$ is used, that is

\begin{equation}
\begin{aligned}
&\int_{\Omega}\ac{\phi}\mr{\phi}_1|\ac{\phi}_1|^{p-2}\partial_{x_i}\ac{\phi}_1dx\\
%\leq &\int_{\Omega}\ac{\phi}\mr{\phi}_1|\ac{\phi}_1|^{\frac{p-2}{2}}|\ac{\phi}_1|^{\frac{p-2}{2}}\partial_{x_i}\ac{\phi}_1dx\\
\leq &\|\ac{\phi}\|_{L^{2p}}\|\mr{\phi}_1\|_{L^{2p}}\|\ac{\phi}_1\|_{L^{p}}^{\frac{p-2}{2}}\|\nabla|\ac{\phi}_1|^{\frac{p}{2}}\|_{L^2}\\
\leq & \nu(\|\nabla|\ac{\phi}_1|^{\frac{p}{2}}\|_{L^2}^2+\|\ac{\phi}_1\|_{L^p}^p)+C\|\ac{\phi}\|_{L^{2p}}^{p}\|\mr{\phi}_1\|_{L^{2p}}^{p}.
\end{aligned}
\end{equation}
The other terms in $G_2$ can be estimated similarly as \eqref{l4.19}. Making use of \eqref{l4.21} in lemma \ref{L4.2}, under the assumption \eqref{nu},
\begin{equation}
\begin{aligned}
\int_{\Omega}G_2dx\leq \nu\big(\|\nabla|\ac{\phi}_1|^{\frac{p}{2}}\|_{L^2}^2+\|\ac{\phi}_1\|_{L^p}^p)+C\varepsilon e^{-\bar{c}t\cdot p}(\|\mr{\phi}_1\|_{L^p}^p+\|\mr{\phi}_1\|_{L^{2p}}^p)+C\varepsilon e^{-\bar{c}t}\|\ac{\phi}_1\|_{L^p}^p.
\end{aligned}
\end{equation}
For $G_3$, it is more easier. We only show the two terms here,
\begin{equation}
\begin{aligned}
\int_{\Omega}(\ac{\phi}+\mr{\phi}\ac{\bar{u}})|\ac{\phi}_1|^{p-2}\partial_{x_i}\ac{\phi}_1dx
\leq \nu(\|\nabla|\ac{\phi}_1|^{\frac{p}{2}}\|_{L^2}^2+\|\ac{\phi}_1\|_{L^p}^p)+C\|\ac{\phi}\|_{L^p}^p+C\varepsilon e^{-\bar{c}t\cdot p}\|\mr{\phi}\|_{L^p}^p.
\end{aligned}
\end{equation}
The estimation for the last term in the righthand side of \eqref{l4.36} is the same as \eqref{l4.13}.
Combining all of this and integrating \eqref{l4.36} over $\Omega$, it yields that
\begin{equation}
\begin{aligned}
\frac{d}{dt}\|\ac{\phi}_1\|_{L^p}^p+b\|\nabla|\ac{\phi}_1|^{\frac{p}{2}}\|_{L^2}^2\leq C e^{-\bar{c}t\cdot p}+(C\varepsilon e^{-\bar{c}t}+\nu)\|\ac{\phi}_1\|_{L^p}^p.
\end{aligned}
\end{equation}
Following the previous steps in Lemma \ref{L4.2}, only need to replace $\phi$ with $\phi_1$, then we  get \eqref{l4.33}.
\end{proof}
Now we obtain the $L^{\infty}$ estimates for $\ac{\phi},$ using \eqref{GN2},
\begin{equation}
\begin{aligned}
\|\ac{\phi}\|_{L^{\infty}(\Omega)}\leq C\sum_{k=0}^{n-1}\|\nabla \ac{\phi}\|_{L^{r_k}(\Omega)}^{\theta_k}\|\ac{\phi}\|_{L^{q_k}(\Omega)}^{1-\theta_k},
\end{aligned}
\end{equation}
where $0=(\frac{1}{r_k}-\frac{1}{k+1})\theta_k+\frac{1}{q_k}(1-\theta_k)$ and $\max\{k+1,2\}\leq r_k<+\infty$ and $1\leq q_k<+\infty$ for $k=0,1,...,n-1.$ It yields that

\begin{equation}\label{ac}
\|\ac{\phi}\|_{L^{\infty}}\leq C\sum_{k=0}^{n-1} e^{-\bar{c}t\cdot\theta_k}e^{-\bar{c}t(1-\theta_k)}\bigg\{(1+t)^{\epsilon}e^{\bar{c}t}\bigg\}^{(\frac{2}{q_k}-1)(1-\theta_k)}\leq Ce^{-c\theta_k t}
\end{equation}
where $\theta_k>0$.

\subsection{Proof of Proposition \ref{pet2}}

Now we start to give the $L^p$ estimates for the zero mode $\mr{\phi},$ from \eqref{aape},\eqref{zm}, the initial problem of $\mr{\phi}$ is following

\begin{equation}\label{zmi}
\left\{
\begin{aligned}
&\partial_{t}\mr{\phi}+\partial_{x_1}\bigg(\Do(f_1(\bar{u}+\phi)-f_1(\bar{u}))\bigg)=a_{11}\partial_{x_1}^2\mr{\phi}-\mr{J},\\
&\mr{\phi}(x,0)=0.
\end{aligned}
\right.
\end{equation}

\begin{Lemma}[Time decay estimate for $\mr{\phi}$]\label{L4.5}
	
	\begin{equation}\label{l4.50}
	\begin{aligned}
	&\|\mr{\phi}(t)\|_{L^p({\mathbb{R}})}\leq C_{p,\epsilon}(1+t)^{-\frac{1}{2}(1-\frac{1}{p})+\epsilon},\quad  \forall p\in [1,+\infty).
	\end{aligned}
	\end{equation}	

\begin{proof}Rewrite \eqref{zmi} into the new form,
	
\begin{equation}\label{l4.51}
\left\{
\begin{aligned}
&\partial_{t}\mr{\phi}+\partial_{x_1}\bigg(f_1(\mr{\bar{u}}+\mr{\phi})-f_1(\mr{\bar{u}})\bigg)\\
=&a_{11}\partial_{x_1}^2\mr{\phi}-\mr{J}+\partial_{x_1}\bigg(f_1(\mr{\bar{u}}+\mr{\phi})-f_1(\mr{\bar{u}})-\Do(f_1(\bar{u}+\phi)-f_1(\bar{u}))\bigg),\\
&\mr{\phi}(x,0)=0.
\end{aligned}
\right.
\end{equation}
Remember that $\mr{\phi}$ is only related to $x_1$, not related to $x^{'}.$ The $L^p$ estimation for $\mr{\phi}$ is similar as \cite{Y} except for the remainder terms in the right hand-side of \eqref{l4.51}. Multiplying \eqref{l4.51} by $|\mr{\phi}|^{p-2}\mr{\phi}$,  it yields that

\begin{equation}\label{bs}
\begin{aligned}
&\frac{1}{p}\partial_{t}|\mr{\phi}|^{p}+(p-1)\int_{0}^{\mr{\phi}}(f_1^{'}(\mr{\bar{u}}+s)-f_1^{'}(\mr{\bar{u}}))|s|^{p-2}ds\partial_{x_1}\mr{\bar{u}}\\
=&\partial_{x_{1}}(\cdots)-(p-1)a_{11}|\mr{\phi}|^{p-2}\mr\phi_{x_1}^2-|\mr\phi|^{p-2}\mr\phi \mr{J}+|\mr\phi|^{p-2}\partial_{x_1}\mr\phi  R .
\end{aligned}
\end{equation}
where $(\cdots)$ are

\begin{equation}\label{bs1}
\begin{aligned}
&a_{11}\partial_{x_{1}}\mr\phi|\mr\phi|^{p-2}\mr\phi-|\mr{\phi}|^{p-2}\mr{\phi}(f_1(\mr{\bar{u}}+\mr{\phi})-f_1(\mr{\bar{u}}))\\
&+(p-1)\int_{0}^{\mr{\phi}}(f_1(\mr{\bar{u}}+s)-f_1(\mr{\bar{u}}))|s|^{p-2}ds\\
&+|\mr{\phi}|^{p-2}\mr{\phi}\bigg(f_1(\mr{\bar{u}}+\mr{\phi})-f_1(\mr{\bar{u}})-\Do(f_1(\bar{u}+\phi)-f_1(\bar{u}))\bigg),
\end{aligned}
\end{equation}
and

\begin{equation}\label{R}
\begin{aligned}
R=O(1)\Do\big(\ac{\bar{u}}\mr{\phi}+\ac{\phi}\mr{\phi}+f_1^{''}(\xi)\ac{\phi}\big).
\end{aligned}
\end{equation}
Because $\mr{\bar{u}}=\hat{u}+(1-\eta)\mr{\tilde{u}}_-+\eta\mr{\tilde{u}}_+,$

\begin{equation}\label{l4.53}
\begin{aligned}
&\int_{0}^{\mr{\phi}}(f_1^{'}(\mr{\bar{u}}+s)-f_1^{'}(\mr{\bar{u}}))|s|^{p-2}ds\partial_{x_1}\mr{\bar{u}}\\
=&\int_{0}^{\mr{\phi}}(f_1^{'}(\mr{\bar{u}}+s)-f_1^{'}(\mr{\bar{u}}))|s|^{p-2}ds\partial_{x_1}{\hat{u}}+\int_{0}^{\mr{\phi}}(f_1^{'}(\mr{\bar{u}}+s)-f_1^{'}(\mr{\bar{u}}))|s|^{p-2}ds\partial_{x_1}(\mr{\bar{u}}-\hat{u})\\
=&Q_1+Q_2.
\end{aligned}
\end{equation}
It is easy to verify that $Q_1$ is a positive term which satisfies

\begin{equation}\label{Q1}
\begin{aligned}
\int_{\mathbb{R}}Q_1dx_1
\geq &C_p\bigg( \int_{\bar{u}+\mr\phi> 0, \bar{u}>0}|\mr\phi|^{p} \partial_{x_1} \hat{u} \mathrm{d} x_1+\int_{\bar{u}+\mr\phi> 0, \bar{u}\leq 0}\left(\mr{\bar{u}}+\mr\phi\right)^p\partial_{x_1} \hat{u}dx_1\\
&+\int_{\mr{\bar{u}}+\mr\phi\leq 0, \bar{u}> 0}\mr{\bar{u}}^p\partial_{x_{1}}\hat{u}dx_1\bigg):=A(t).\\
\end{aligned}
\end{equation}

As for $Q_2,$ we get

\begin{equation}\label{Q2}
\begin{aligned}
\int_{\mathbb{R}}|Q_2|dx_1\leq C\varepsilon e^{-\bar{c}t}\|\mr{\phi}\|_{L^p(\mathbb{R})}^p.
\end{aligned}
\end{equation}
Using the interpolation inequality (refer to lemma 5.4 in \cite{Y}) and \eqref{Jp},\eqref{nu}, we get
\begin{equation}\label{l4.54}
\begin{aligned}
&\int_{\Omega}|\mr\phi|^{p-2}\mr\phi \mr{J}dx\leq  \|\mr{\phi}\|_{L^{\infty}(\mathbb{R})}^{p-1}\|\mr{J}\|_{L^1(\mathbb{R})}\\
\leq & C_p\|\partial_{x_1}|\mr{\phi}|^{\frac{p}{2}}\|_{L^2(\mathbb{R})}^{\frac{2(p-1)}{p+1}}\|\mr{\phi}\|_{L^1}^{\frac{p-1}{p+1}}\|\mr{J}\|_{L^1(\mathbb{R})}\\
\leq & \nu\|\partial_{x_1}|\mr{\phi}|^{\frac{p}{2}}\|_{L^2(\mathbb{R})}^2+C_p\|\mr{\phi}\|_{L^1}^{\frac{p-1}{2}}\|\mr{J}\|_{L^1(\mathbb{R})}^{\frac{p+1}{2}}\\
\leq &\nu\|\partial_{x_1}|\mr{\phi}|^{\frac{p}{2}}\|_{L^2(\mathbb{R})}^2+C_{p,\epsilon}(1+t)^{-\frac{p+1}{2}}(1+t)^{p\epsilon}.
\end{aligned}
\end{equation}

For the remainder term, similar as before
\begin{equation}\label{l4.55}
\begin{aligned}
&\int_{\mathbb{R}}|\mr\phi|^{p-2}\partial_{x_1}\mr\phi  R dx_1\\
\leq &C\int_{\mathbb{R}}|\mr\phi|^{p-2}\partial_{x_1}\mr\phi\Do\big(\ac{\bar{u}}\mr{\phi}+\ac{\phi}\mr{\phi}+f_1^{''}(\xi)\ac{\phi}\big)dx_1\\
\leq &\nu \|\partial_{x_1}|\mr{\phi}|^{\frac{p}{2}}\|_{L^2(\mathbb{R})}^2+ Ce^{-\bar{c}t}\|\mr{\phi}\|_{L^p}^p+C\|\mr{\phi}\|_{L^p}^{p-2}\|\ac{\phi}\|_{L^p}^2\\
\leq &\nu\|\partial_{x_1}|\mr{\phi}|^{\frac{p}{2}}\|_{L^2(\mathbb{R})}^2+ Ce^{-\bar{c}t}\|\mr{\phi}\|_{L^p}^p+Ce^{-\bar{c}t}.
\end{aligned}
\end{equation}
Combining all of this and integrating \eqref{bs} over $\mathbb{R}$, it yields that
\begin{equation}\label{l4.56}
\begin{aligned}
\frac{d}{dt}\|\mr{\phi}\|_{L^p(\mathbb{R})}^p+b\|\partial_{x_1}|\mr{\phi}|^{\frac{p}{2}}\|_{L^2(\mathbb{R})}^2+A(t)\leq Ce^{-\bar{c}t}\|\mr{\phi}\|_{L^p}^p+C_{p,\epsilon}(1+t)^{-\frac{p+1}{2}}(1+t)^{p\epsilon}.
\end{aligned}
\end{equation}
Thus, multiplying \eqref{l4.56} by $(1+t)^{\alpha}$ and integrating the resulting equation over $[0,T]$ yields that

\begin{equation}\label{l4.57}
\begin{aligned}
&(1+t)^{\alpha}\|\mr{\phi}\|_{L^p(\mathbb{R})}^p+C\int_{0}^{T}(1+t)^{\alpha}\big(\|\partial_{x_1}|\mr{\phi}|^{\frac{p}{2}}\|_{L^2(\mathbb{R})}^2+A(t)\big)dt\\
\leq & C\int_{0}^{T}(1+t)^{\alpha-1}\|\mr{\phi}\|_{L^p}^pdt+C_{p,\epsilon}\int_{0}^{T}(1+t)^{\alpha-\frac{p+1}{2}}(1+t)^{p\epsilon}dt\\
\leq &C_{p,\epsilon}(1+T)^{\alpha-\frac{p-1}{2}+p\epsilon},
\end{aligned}
\end{equation}
where
\begin{equation}\label{14.58}
\begin{aligned}
&\int_{0}^{T}(1+t)^{\alpha-1}\|\mr{\phi}\|_{L^p(\mathbb{R})}^pdt\leq \int_{0}^{T}(1+t)^{\alpha-1}\|\partial_{x_1}|\mr{\phi}|^{\frac{p}{2}}\|_{L^2(\mathbb{R})}^{\frac{2(p-1)}{p+1}}\|\mr{\phi}\|_{L^1(\mathbb{R})}^{\frac{2p}{p+1}}dt\\
\leq &\nu\int_{0}^{T}(1+t)^{\alpha}\|\partial_{x_1}|\mr{\phi}|^{\frac{p}{2}}\|_{L^2(\mathbb{R})}^2dt+C\int_{0}^{T}(1+t)^{\alpha-\frac{p+1}{2}}\|\mr{\phi}\|_{L^1}^{p}dt.
\end{aligned}
\end{equation}
Finally we can get \eqref{l4.50}.

\end{proof}

\end{Lemma}

\begin{Lemma}[Time decay estimate for $\partial_{x_1}\mr{\phi}, 2\leq p<+\infty$]\label{L4.6}
	
	\begin{equation}\label{l4.6}
	\begin{aligned}
	\|\partial_{x_1}\mr{\phi}\|_{L^p(\mathbb{R})}\leq C_{p,\epsilon}(1+t)^{-\frac{1}{2}(1-\frac{1}{p})+\epsilon},\ \forall p\in[2,+\infty).
	\end{aligned}
	\end{equation}

\end{Lemma}

\begin{proof}Step 1: Remember  $\mr{\phi}_1:=\partial_{x_1}\mr{\phi},$ taking  the derivative on \eqref{zmi} with respect to $x_1$ yields that

\begin{equation}\label{l4.61}
\begin{aligned}
&\partial_{t}\mr{\phi}_1+\partial_{x_{1}}\bigg(f_1'(\mr{\bar{u}})\mr{\phi}_1\bigg)+\partial_{x_1}\bigg((f_1^{'}(\mr{\bar{u}}+\mr{\phi})-f_1^{'}(\mr{\bar{u}}))(\partial_1\mr{\bar{u}}+\mr{\phi}_1)\bigg)\\
=&a_{11}\partial_{x_1}^2\mr{\phi}_1-\partial_1\mr{J}+\partial_{x_1}^2\bigg(f_1(\mr{\bar{u}}+\mr{\phi})-f_1(\mr{\bar{u}})-\Do(f_1(\bar{u}+\phi)-f_1(\bar{u}))\bigg).
\end{aligned}
\end{equation}
Multiplying \eqref{l4.61}  by $|\mr{\phi}_1|^{p-2}\mr{\phi}_1$,  it yields that

\begin{equation}\label{l4.62}
\begin{aligned}
&\frac{1}{p}\pt|\mr{\phi}_1|^p+(p-1) a_{11}|\mr{\phi}_1|^{p-2}\mr{\phi}_{1x_1}^2+\partial_{x_{1}}(\cdots)\\
=&H-\partial_{x_{1}}\mr{J}(|\mr{\phi}_1|^{p-2}\mr{\phi}_{1}),
\end{aligned}
\end{equation}
where the term $H$ can be divided into following situations,

\begin{equation}\label{H}
\begin{aligned}
H_1=&\bigg(f_i'(\mr{\bar{u}})\mr{\phi}_1\bigg)\partial_{x_1}(|\mr{\phi}_1|^{p-2}\mr{\phi}_1),\\
H_2=&(f_1^{'}(\mr{\bar{u}}+\mr{\phi})-f_1^{'}(\mr{\bar{u}}))(\partial_1\mr{\bar{u}}+\mr{\phi}_1)\partial_{x_1}(|\mr{\phi}_1|^{p-2}\mr{\phi}_1),\\
H_3=&\bigg\{\ac{\phi}_1+(\partial_1\ac{\bar{u}}+\ac{\bar{u}}+\ac{\phi})(\mr{\phi}+\ac{\phi})+\partial_1\mr{\bar{u}}\ac{\phi}+\mr{\phi}_1(\ac{\bar{u}}+\ac{\phi})\bigg\}\partial_{x_1}(|\mr{\phi}_1|^{p-2}\mr{\phi}_1).
\end{aligned}
\end{equation}
For $H_1$, the estimate is same as $G_1$\eqref{l4.16}. That is

\begin{equation}\label{H-1}
\begin{aligned}
\int_{\mathbb{R}}H_1dx_1
=&\int_{\mathbb{R}}\partial_{x_1}\bigg(\frac{p-1}{p}f_1'(\mr{\bar{u}})|\mr{\phi}_1|^{p}\bigg)-\frac{p-1}{p}f^{''}_1(\mr{\bar{u}})|\mr{\phi}_1|^{p}\partial_{x_{1}}\hat{u}dx_1\\
&-\frac{p-1}{p}f^{''}_1(\mr{\bar{u}})|\mr{\phi}_1|^{p}\partial_{x_{1}}\bigg((1-\eta)\mr{\tilde{u}}_-+\eta\mr{\tilde{u}}_+\bigg)dx_1\\
\leq &Ce^{-ct}\|\mr{\phi}_1\|_{L^p}^p.
\end{aligned}
\end{equation}
For $H_2$,

\begin{equation}\label{H2}
\begin{aligned}
\int_{\mathbb{R}}H_2dx_1\leq & \|\mr{\phi}_1\|_{L^p}^{\frac{p-2}{2}}\|\partial_{x_1}|\mr{\phi}_1|^{\frac{p}{2}}\|_{L^2}(\|(Q_1,Q_2)\|_{L^1}^{\frac{1}{p}})+\|\partial_{x_1}|\mr{\phi}_1|^{\frac{p}{2}}\|_{L^2}\|\mr{\phi}\|_{L^{\infty}}\|\mr{\phi}_1\|_{L^p}^{\frac{p}{2}}\\
\leq &\nu\|\partial_{x_1}|\mr{\phi}_1|^{\frac{p}{2}}\|_{L^2}^2+\|(Q_1,Q_2)\|_{L^1}+C\|\mr{\phi}\|_{L^{\infty}}^2\|\mr{\phi}_1\|_{L^p}^p.
\end{aligned}
\end{equation}
For $H_3$, similar as before, we have
\begin{equation}\label{H3}
\begin{aligned}
\int_{\mathbb{R}}H_3dx_1\leq &\nu\|\partial_{x_1}|\mr{\phi}_1|^{\frac{p}{2}}\|_{L^2}^2+Ce^{-ct\cdot p}+Ce^{-ct}\|\mr{\phi}_1\|_{L^p}^p+\|\ac{\phi}\|_{L^{\infty}}\|\mr{\phi}_1\|_{L^p}^p,
\end{aligned}
\end{equation}
where we use previous lemmas \ref{L4.2}-\ref{L4.5}. Combining all of this and integrating \eqref{l4.62} over $\mathbb{R},$  we have
\begin{equation}\label{l4.65}
\begin{aligned}
&\frac{d}{dt}\|\mr{\phi}_1\|_{L^p}^p+b\|\partial_{x_1}|\mr{\phi}_1|^{\frac{p}{2}}\|_{L^2}^2\leq C e^{-ct\cdot p}+C(e^{-ct}+\|\mr{\phi}\|_{L^{\infty}}^2)\|\mr{\phi}_1\|_{L^p}^p\\
&+(\|\partial_1\mr{J}\|_{L^1}\|\mr{\phi}_1\|_{L^{\infty}}^{p-1}+\|(Q_1,Q_2)\|_{L^1}).
\end{aligned}
\end{equation}
Step 2: For $p=2,$ from \eqref{Jp}, \eqref{l4.57},
\begin{equation}\label{l4.66}
\begin{aligned}
\int_{0}^{t}\|\mr{\phi}_1\|_{L^2}^2+\|\partial_{x_{1}}\mr{J}_1\|_{L^2}^2+\|(Q_1,Q_2)\|_{L^1}d\tau\leq C(1+t)^{-\frac{1}{2}+\epsilon}.
\end{aligned}
\end{equation}
Similar as before, we could get
\begin{equation}\label{l4.67}
\begin{aligned}
&\|\mr{\phi}_1\|_{L^2(\mathbb{R})}\leq C_{p,\epsilon}(1+t)^{-\frac{1}{4}+\epsilon}.
\end{aligned}
\end{equation}	
Then the interpolation inequality for any $\theta>0$ and \eqref{l4.50}, \eqref{l4.67} gives that
\begin{equation}\label{mr}
\begin{aligned}
\|\mr{\phi}\|_{L^{\infty}(\mathbb{R})}\leq C_{p,\theta}\|\mr{\phi}\|_{L^p}^{1-\theta}\|\partial_{x_1}\mr{\phi}\|_{L^2}^{\theta}\leq C_{p,\theta}(1+t)^{-\frac{1}{2}+\epsilon}.
\end{aligned}
\end{equation}
When $p>2$,  by using \eqref{Jp}, \eqref{l4.50} and the interpolation inequality(refer to lemma 4.2 in \cite{Y}),
\begin{equation}
\begin{aligned}
&\|\mr{\phi}\|_{L^{\infty}}^2\|\mr{\phi}_1\|_{L^p}^p\leq C(1+t)^{-1+\epsilon}\|\mr{\phi}_1\|_{L^p}^p,\\
&(1+t)^{-1+\epsilon}\|\mr{\phi}_1\|_{L^p}^{p}\leq C(1+t)^{-1+\epsilon}\|\partial_{x_1}|\mr{\phi}_1|^{\frac{p}{2}}\|_{L^2}^{\frac{2p}{p+2}}\|\mr{\phi}\|_{L^p}^{\frac{2p}{p+2}}\\
\leq & \nu\|\partial_{x_1}|\mr{\phi}_1|^{\frac{p}{2}}\|_{L^2}^2+C(1+t)^{-\frac{p+2}{2}(1-\epsilon)}\|\mr{\phi}\|_{L^p}^{p}\\
\leq &\nu\|\partial_{x_1}|\mr{\phi}_1|^{\frac{p}{2}}\|_{L^2}^2+C(1+t)^{-\frac{p+2}{2}(1-\epsilon)}(1+t)^{-\frac{p-1}{2}+p\epsilon}\\
\leq & \nu\|\partial_{x_1}|\mr{\phi}_1|^{\frac{p}{2}}\|_{L^2}^2+C(1+t)^{-(p+\frac{1}{2}-p\epsilon)}
\end{aligned}
\end{equation}

\begin{equation}
\begin{aligned}
\|\partial_1\mr{J}\|_{L^1}\|\mr{\phi}_1\|_{L^{\infty}}^{p-1}\leq &C(1+t)^{-(1-\epsilon)}\|\partial_{x_1}|\mr{\phi}_1|^{\frac{p}{2}}\|_{L^2}^{\frac{2(p-1)}{p+2}}\|\mr{\phi}_1\|_{L^2}^{\frac{2(p-1)}{p+2}}\\
\leq &\nu\|\partial_{x_1}|\mr{\phi}_1|^{\frac{p}{2}}\|_{L^2}^2+C(1+t)^{-\frac{p+1}{2}+p\epsilon}
\end{aligned}
\end{equation}
We can finally get \eqref{l4.6}.
\end{proof}
Now  \eqref{ac}  and \eqref{mr} tell us
\begin{equation}
\begin{aligned}
\|\phi\|_{L^{\infty}}\leq \|\ac{\phi}\|_{L^{\infty}}+\|\mr{\phi}\|_{L^{\infty}}\leq C(1+t)^{-\frac{1}{2}+\epsilon}.
\end{aligned}
\end{equation}
We prove theorem \ref{mt}.

\end{document}